\documentclass[12pt]{article}
\usepackage{fullpage}
\usepackage{epsfig}
\usepackage{amssymb}
\usepackage{amsmath}
\usepackage{amsfonts}

\usepackage[ansinew]{inputenc}

\usepackage{latexsym}
\usepackage{amsthm}
\usepackage{amsfonts,amssymb,amsmath, amscd}
\usepackage{mathrsfs}
\usepackage{bbm}

\newcommand {\debeq}	{\begin{eqnarray*}}
\newcommand {\fineq}	{\end{eqnarray*}}
\newcommand {\lbd}	{\lambda}
\newcommand     {\eps}  {\epsilon}
\newcommand     {\vareps}       {\varepsilon}

\newcommand	{\intgen}	
{\int_0^\infty}

\newcommand	{\PP}{\mathbb{P}}
\newcommand	{\EE}{\mathbb{E}}

\newcommand {\cadlag}	{c{\`a}dl{\`a}g}

\newtheorem{thm}{Theorem}
\newtheorem{lem}[thm]{Lemma}

\newtheorem {rem}{Remark}
\newtheorem{prop}[thm]{Proposition}
\newtheorem{cor}[thm]{Corollary}

\newcommand	{\indic}	[1]
{{\bf{1}}_{\{#1\}}}

\begin{document}

\title{Splitting trees stopped when the first clock rings and Vervaat's transformation}
\author{Amaury Lambert$^{1}$and Pieter Trapman$^{2}$}
\footnotetext[1]{Laboratoire de Probabilit{\'e}s et Mod{\`e}les Al{\'e}atoires UMR 7599 CNRS,  Case courrier 188, UPMC Univ Paris 06, 4, Place Jussieu, F-75252 Paris Cedex 05, France}
\footnotetext[2]{Stockholm University, Department of Mathematics, 106 91 Stockholm, Sweden.}
\date{\today}
\maketitle

\begin{abstract}
\noindent
We consider a branching population where individuals have i.i.d.\ life lengths (not necessarily exponential) and constant birth rate. We let $N_t$ denote the population size at time $t$. 
We further assume that all individuals, at birth time, are equipped with  independent exponential clocks with parameter $\delta$. We are interested in the genealogical tree stopped at the first time $T$ when one of those clocks rings. This question has applications in epidemiology, in population genetics, in ecology and in queuing theory.

We show that conditional on $\{T<\infty\}$, the joint law of $(N_T, T, X^{(T)})$, where $X^{(T)}$ is the jumping contour process of the tree truncated at time $T$, is equal to that of $(M, -I_M, Y_M')$ conditional on $\{M\not=0\}$, where : $M+1$ is the number of visits of 0, before some single independent exponential clock $\mathbf{e}$ with parameter $\delta$ rings, by some specified L{\'e}vy process $Y$ without negative jumps reflected below its supremum; $I_M$ is the infimum of the path $Y_M$ defined as $Y$ killed at its last 0 before $\mathbf{e}$; $Y_M'$ is the Vervaat transform of $Y_M$.

This identity yields an explanation for the geometric distribution of $N_T$ \cite{K,T} and has numerous other applications. In particular, conditional on $\{N_T=n\}$, and also on $\{N_T=n, T<a\}$, the ages and residual lifetimes of the $n$ alive individuals at time $T$ are i.i.d.\ and independent of $n$. We provide explicit formulae for this distribution and give a more general application to outbreaks of antibiotic-resistant bacteria in the hospital.


\end{abstract}  	

\section{Introduction}

We consider a population of particles behaving independently from one another, where each particle gives birth at constant rate $b>0$ during its lifetime (inter-birth durations are i.i.d.\ exponential random variables with parameter $b$), and where lifetime durations are i.i.d.\ on $(0,+\infty]$ (some particles may have infinite lifetimes) with probability distribution $\mu$ (not necessarily exponential). 
 
The genealogical trees  that we consider here are usually called \emph{splitting trees} \cite{GK}. We define the \emph{lifespan measure} as the measure on $(0,+\infty]$ with total mass $b$ simply defined as $\pi:=b\mu$. 
 
The process $(N_t;t\ge 0)$ giving the number of extant particles at time $t$, belongs to a wide class of branching processes called \emph{Crump--Mode--Jagers processes}. Actually, the processes we consider are homogeneous (constant birth rate) and binary (one birth at a time) but are more general than classical (simple) 
 birth--death processes \cite{Grim92} 
 in that the lifetime durations may follow a general distribution. 

In addition, we assume that each particle is independently equipped with a random exponential clock with parameter $\delta>0$. We are interested in the first time $T$ when one of those clocks rings, called \emph{detection time}. See Figure \ref{fig:1} for a realisation of a splitting tree with individual clocks. Note that on the extinction event, $T$ can be infinite (no clock rings) with positive probability. 

\footnotetext{Keywords: branching process; splitting tree; Crump--Mode--Jagers process; contour process; L{\'e}vy process; scale function; resolvent; age and residual lifetime; undershoot and overshoot; Vervaat's transformation; sampling; detection; epidemiology; processor-sharing.}
\footnotetext{2000 Mathematical Subject Classification: Primary: 60J80; Secondary: 92D10; 92D25; 92D30; 92D40; 60J85; 60G17; 60G51; 60G55; 60K15; 60K25.}

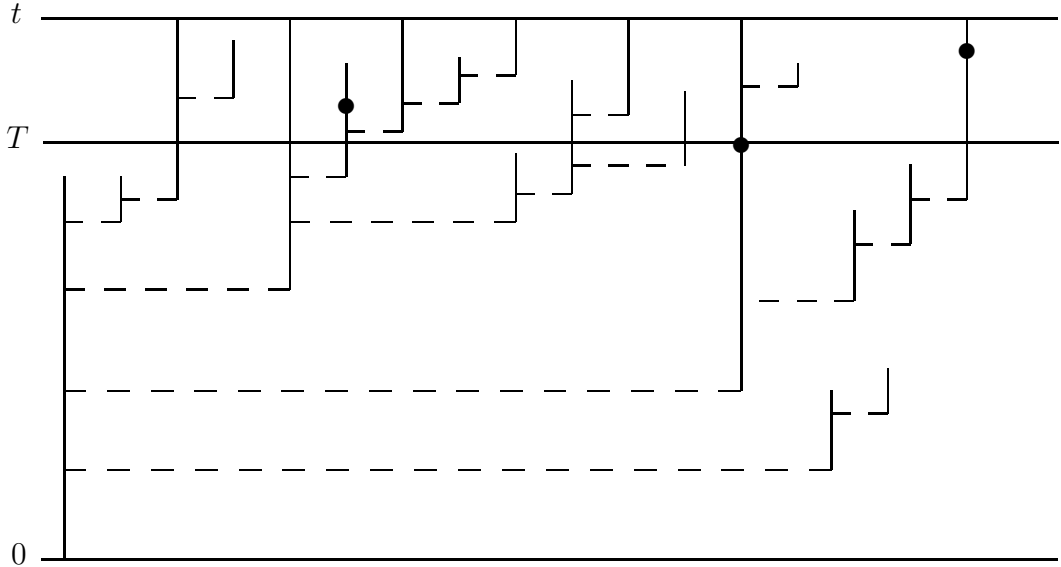
\begin{figure}[ht]
\label{fig:1}
\unitlength 3mm 
\linethickness{0.5pt}
\ifx\plotpoint\undefined\newsavebox{\plotpoint}\fi 
\begin{picture}(48.063,25.25)(3,-2)
\put(6,23){\line(1,0){45}}
\put(6.063,17.5){\line(1,0){45}}
\put(6,-1){\line(1,0){45}}
\put(17,23){\line(0,-1){12}}
\put(37,23){\line(0,-1){16.5}}
\put(19.5,19.125){\circle*{.707}}
\put(37,17.375){\circle*{.707}}
\put(47,21.563){\circle*{.707}}
\put(16.971,10.971){\line(-1,0){.9091}}
\put(15.153,10.971){\line(-1,0){.9091}}
\put(13.334,10.971){\line(-1,0){.9091}}
\put(11.516,10.971){\line(-1,0){.9091}}
\put(9.698,10.971){\line(-1,0){.9091}}
\put(7.88,10.971){\line(-1,0){.9091}}
\put(36.971,6.471){\line(-1,0){.9677}}
\put(35.035,6.471){\line(-1,0){.9677}}
\put(33.1,6.471){\line(-1,0){.9677}}
\put(31.164,6.471){\line(-1,0){.9677}}
\put(29.229,6.471){\line(-1,0){.9677}}
\put(27.293,6.471){\line(-1,0){.9677}}
\put(25.358,6.471){\line(-1,0){.9677}}
\put(23.422,6.471){\line(-1,0){.9677}}
\put(21.487,6.471){\line(-1,0){.9677}}
\put(19.551,6.471){\line(-1,0){.9677}}
\put(17.616,6.471){\line(-1,0){.9677}}
\put(15.68,6.471){\line(-1,0){.9677}}
\put(13.745,6.471){\line(-1,0){.9677}}
\put(11.809,6.471){\line(-1,0){.9677}}
\put(9.874,6.471){\line(-1,0){.9677}}
\put(7.938,6.471){\line(-1,0){.9677}}
\put(14.5,19.5){\line(0,1){2.5}}
\put(14.471,19.471){\line(-1,0){.8333}}
\put(12.804,19.471){\line(-1,0){.8333}}
\put(16.971,13.971){\line(1,0){.9091}}
\put(18.789,13.971){\line(1,0){.9091}}
\put(20.607,13.971){\line(1,0){.9091}}
\put(22.425,13.971){\line(1,0){.9091}}
\put(24.243,13.971){\line(1,0){.9091}}
\put(26.062,13.971){\line(1,0){.9091}}
\put(27,14){\line(0,1){3}}
\put(21.971,17.971){\line(-1,0){.8333}}
\put(20.304,17.971){\line(-1,0){.8333}}
\put(22,23){\line(0,-1){5}}
\put(32,23){\line(0,-1){4.25}}
\put(29.5,20.25){\line(0,-1){5}}
\put(29.471,15.221){\line(-1,0){.8333}}
\put(27.804,15.221){\line(-1,0){.8333}}
\put(31.971,18.721){\line(-1,0){.8333}}
\put(30.304,18.721){\line(-1,0){.8333}}
\put(26.971,20.471){\line(-1,0){.8333}}
\put(25.304,20.471){\line(-1,0){.8333}}
\put(27,23){\line(0,-1){2.5}}
\put(24.5,21.25){\line(0,-1){2}}
\put(24.471,19.221){\line(-1,0){.8333}}
\put(22.804,19.221){\line(-1,0){.8333}}
\put(29.471,16.471){\line(1,0){.8333}}
\put(31.137,16.471){\line(1,0){.8333}}
\put(32.804,16.471){\line(1,0){.8333}}
\put(34.5,16.5){\line(0,1){3.25}}
\put(36.971,19.971){\line(1,0){.8333}}
\put(38.637,19.971){\line(1,0){.8333}}
\put(39.5,20){\line(0,1){1}}
\put(44.5,16.5){\line(0,-1){3.5}}
\put(44.471,14.971){\line(1,0){.8333}}
\put(46.137,14.971){\line(1,0){.8333}}
\put(47,23){\line(0,-1){8}}
\put(4.875,23.25){\makebox(0,0)[cc]{$t$}}
\put(4.938,17.75){\makebox(0,0)[cc]{$T$}}
\put(12,23){\line(0,-1){8}}
\put(9.5,16){\line(0,-1){2}}
\put(42,14.5){\line(0,-1){4}}
\put(11.971,14.971){\line(-1,0){.8333}}
\put(10.304,14.971){\line(-1,0){.8333}}
\put(9.471,13.971){\line(-1,0){.8333}}
\put(7.804,13.971){\line(-1,0){.8333}}
\put(44.471,12.971){\line(-1,0){.8333}}
\put(42.804,12.971){\line(-1,0){.8333}}
\put(41.971,10.471){\line(-1,0){.8333}}
\put(40.304,10.471){\line(-1,0){.8333}}
\put(38.637,10.471){\line(-1,0){.8333}}
\put(19.5,21){\line(0,-1){5}}
\put(19.471,15.971){\line(-1,0){.8333}}
\put(17.804,15.971){\line(-1,0){.8333}}
\put(6.971,2.971){\line(1,0){.9722}}
\put(8.915,2.971){\line(1,0){.9722}}
\put(10.86,2.971){\line(1,0){.9722}}
\put(12.804,2.971){\line(1,0){.9722}}
\put(14.748,2.971){\line(1,0){.9722}}
\put(16.693,2.971){\line(1,0){.9722}}
\put(18.637,2.971){\line(1,0){.9722}}
\put(20.582,2.971){\line(1,0){.9722}}
\put(22.526,2.971){\line(1,0){.9722}}
\put(24.471,2.971){\line(1,0){.9722}}
\put(26.415,2.971){\line(1,0){.9722}}
\put(28.36,2.971){\line(1,0){.9722}}
\put(30.304,2.971){\line(1,0){.9722}}
\put(32.248,2.971){\line(1,0){.9722}}
\put(34.193,2.971){\line(1,0){.9722}}
\put(36.137,2.971){\line(1,0){.9722}}
\put(38.082,2.971){\line(1,0){.9722}}
\put(40.026,2.971){\line(1,0){.9722}}
\put(41,3){\line(0,1){3.5}}
\put(43.5,7.5){\line(0,-1){2}}
\put(43.471,5.471){\line(-1,0){.8333}}
\put(41.804,5.471){\line(-1,0){.8333}}
\put(7,16){\line(0,-1){17}}
\put(5,-.75){\makebox(0,0)[cc]{$0$}}
\end{picture}
\caption{A realisation of a splitting tree with individual exponential clocks. Time flows from bottom to top; horizontal dashed lines show filiation. Solid dots show individual ringing clocks. The time $T$ when the first clock rings is indicated. Here $N_T=7$.}

\end{figure}

This question has applications in population genetics and in ecology \cite{CL1, CL2, L2, L-JMB} ($T$ is then the first time when a new mutant or a new species arises), in queuing theory \cite{K, Grishechkin, LSZ} (because $N$ is a time-changed processor-sharing queue, and then in the new timescale, $T$ is a single, independent exponential clock), and in epidemiology \cite{BWTB, T} ($T$ is then the first detection time of the epidemics). In this last setting, ages of individuals in the population at $T$ are the times since infection of infectives in the detected outbreak, and in the last section we see how this information can be enlarged with more easily available data such as the length of stay in the hospital up to time $T$. 

Our main result is to characterise, on the event $\{T<\infty\}$, the joint law of $(N_T, T, X^{(T)})$, where $X^{(T)}$ is the jumping contour process of the tree truncated at time $T$, in terms of the Vervaat transform of the path of the (reflected) L{\'e}vy process $X$ with jump measure $\pi$ and slope $-1$. In particular, we recover the known fact \cite{K,T} that conditional on $\{T<\infty\}$, $N_T$ is geometrically distributed, and we characterise the joint law of $T$ and $N_T$ in terms of (joint) Laplace transforms of some hitting times of $X$. As a further example, restricting the main identity to the undershoots and overshoots of $X$ whenever it crosses 0, we get the following application. Conditional on $\{N_T=n\}$, and also on $\{N_T=n, T<a\}$, the ages and residual lifetimes of the $n$ alive individuals at time $T$ are i.i.d.\ and independent of $n$ and follow the bivariate law of $(Und, Ove)$ (resp. $(Und_a, Ove_a)$) defined hereafter. The pair $(Und, Ove)$ (resp. $(Und_a, Ove_a)$) is the undershoot and overshoot of the jump across $0$ of $X$, at its first hitting time $\tau_0^+$ of $(0,+\infty]$, conditional on $\tau_0^+$ smaller than some independent exponential time with parameter $\delta$ (resp. and $\inf_{0\le s \le \tau_0^+} X_s >-a$). In the epidemics model, these statements are extended by taking into account, in addition to the age and residual lifetime (of individual infection at time $T$), the length of stay in the hospital up to infection time. In all cases, explicit formulae are also provided for these laws.


\section{Splitting trees and L{\'e}vy processes}

We assume that splitting trees are started with one unique progenitor born at time 0. We denote by $\PP$ their law, and the subscript $s$ in $\PP_s$ means conditioning on the lifetime of the progenitor being $s$. Of course if $\PP$ bears no subscript, this means that the lifetime of the progenitor follows the usual distribution $\mu$.

In \cite{L}, for $t>0$, the first author has 
considered the so-called jumping chronological contour process (JCCP), here denoted $X^{(t)}$, of the splitting tree truncated up to height (time) $t$, which starts at $s\wedge t$ (here and in what follows $x \wedge y$ 
denotes the minimum of $x$ and $y$), where $s$ is the time of death 
of the progenitor, visits all existence times (smaller than $t$) of all individuals exactly once and terminates at 0 (see Figure \ref{fig : jccp}).

\begin{figure}[ht]
\unitlength 1.9mm 
\linethickness{0.2pt}
\begin{picture}(67.13,70)(0,0)
\thicklines
\put(7,40){\line(0,1){12}}
\thinlines
\put(10,49){\line(1,0){.07}}
\put(6.93,48.93){\line(1,0){.75}}
\put(8.43,48.93){\line(1,0){.75}}
\put(0,0){}
\thicklines
\put(10,49){\line(0,1){8}}
\thinlines
\put(13,55){\line(1,0){.07}}
\put(9.93,54.93){\line(1,0){.75}}
\put(11.43,54.93){\line(1,0){.75}}
\put(0,0){}
\thicklines
\put(13,55){\line(0,1){5}}
\thinlines
\put(15,53){\line(1,0){.07}}
\put(9.93,52.93){\line(1,0){.833}}
\put(11.6,52.93){\line(1,0){.833}}
\put(13.26,52.93){\line(1,0){.833}}
\put(0,0){}
\thicklines
\put(15,53){\line(0,1){4}}
\put(18,54){\line(0,1){5}}
\put(24,44){\line(0,1){6}}
\thinlines
\put(27,48){\line(1,0){.07}}
\put(23.93,47.93){\line(1,0){.75}}
\put(25.43,47.93){\line(1,0){.75}}
\put(0,0){}
\thicklines
\put(27,48){\line(0,1){15}}
\thinlines
\put(30,61){\line(1,0){.07}}
\put(26.93,60.93){\line(1,0){.75}}
\put(28.43,60.93){\line(1,0){.75}}
\put(0,0){}
\thicklines
\put(30,61){\line(0,1){3}}
\thinlines
\put(18,54){\line(1,0){.07}}
\put(14.93,53.93){\line(1,0){.75}}
\put(16.43,53.93){\line(1,0){.75}}
\put(0,0){}
\put(33,62){\line(1,0){.07}}
\put(29.93,61.93){\line(1,0){.75}}
\put(31.43,61.93){\line(1,0){.75}}
\put(0,0){}
\thicklines
\put(33,62){\line(0,1){4}}
\thinlines
\put(36,65){\line(1,0){.07}}
\put(32.93,64.93){\line(1,0){.75}}
\put(34.43,64.93){\line(1,0){.75}}
\put(0,0){}
\thicklines
\put(36,65){\line(0,1){3}}
\thinlines
\put(38,64){\line(1,0){.07}}
\put(32.93,63.93){\line(1,0){.833}}
\put(34.6,63.93){\line(1,0){.833}}
\put(36.26,63.93){\line(1,0){.833}}
\put(0,0){}
\thicklines
\put(38,64){\line(0,1){2}}
\thinlines
\put(35,57){\line(1,0){.07}}
\put(26.93,56.93){\line(1,0){.889}}
\put(28.71,56.93){\line(1,0){.889}}
\put(30.49,56.93){\line(1,0){.889}}
\put(32.26,56.93){\line(1,0){.889}}
\put(34.04,56.93){\line(1,0){.889}}
\put(0,0){}
\thicklines
\put(35,57){\line(0,1){4}}
\thinlines
\put(40,51){\line(1,0){.07}}
\put(26.93,50.93){\line(1,0){.929}}
\put(28.79,50.93){\line(1,0){.929}}
\put(30.64,50.93){\line(1,0){.929}}
\put(32.5,50.93){\line(1,0){.929}}
\put(34.36,50.93){\line(1,0){.929}}
\put(36.22,50.93){\line(1,0){.929}}
\put(38.07,50.93){\line(1,0){.929}}
\put(0,0){}
\thicklines
\put(40,51){\line(0,1){8}}
\thinlines
\put(43,57){\line(1,0){.07}}
\put(39.93,56.93){\line(1,0){.75}}
\put(41.43,56.93){\line(1,0){.75}}
\put(0,0){}
\thicklines
\put(43,57){\line(0,1){4}}
\thinlines
\put(45,55){\line(1,0){.07}}
\put(39.93,54.93){\line(1,0){.833}}
\put(41.6,54.93){\line(1,0){.833}}
\put(43.26,54.93){\line(1,0){.833}}
\put(0,0){}
\thicklines
\put(45,55){\line(0,1){4}}
\thinlines
\put(48,58){\line(1,0){.07}}
\put(44.93,57.93){\line(1,0){.75}}
\put(46.43,57.93){\line(1,0){.75}}
\put(0,0){}
\thicklines
\put(48,58){\line(0,1){7}}
\thinlines
\put(53,61){\line(1,0){.07}}
\put(47.93,60.93){\line(1,0){.833}}
\put(49.6,60.93){\line(1,0){.833}}
\put(51.26,60.93){\line(1,0){.833}}
\put(0,0){}
\put(55,59){\line(1,0){.07}}
\put(47.93,58.93){\line(1,0){.875}}
\put(49.68,58.93){\line(1,0){.875}}
\put(51.43,58.93){\line(1,0){.875}}
\put(53.18,58.93){\line(1,0){.875}}
\put(0,0){}
\thicklines
\put(55,59){\line(0,1){3}}
\thinlines
\put(57,52){\line(1,0){.07}}
\put(39.93,51.93){\line(1,0){.944}}
\put(41.82,51.93){\line(1,0){.944}}
\put(43.71,51.93){\line(1,0){.944}}
\put(45.6,51.93){\line(1,0){.944}}
\put(47.49,51.93){\line(1,0){.944}}
\put(49.37,51.93){\line(1,0){.944}}
\put(51.26,51.93){\line(1,0){.944}}
\put(53.15,51.93){\line(1,0){.944}}
\put(55.04,51.93){\line(1,0){.944}}
\put(0,0){}
\thicklines
\put(57,52){\line(0,1){5}}
\thinlines
\put(60,55){\line(1,0){.07}}
\put(56.93,54.93){\line(1,0){.75}}
\put(58.43,54.93){\line(1,0){.75}}
\put(0,0){}
\thicklines
\put(60,55){\line(0,1){3}}
\thinlines
\put(62,46){\line(1,0){.07}}
\put(23.93,45.93){\line(1,0){.974}}
\put(25.88,45.93){\line(1,0){.974}}
\put(27.83,45.93){\line(1,0){.974}}
\put(29.78,45.93){\line(1,0){.974}}
\put(31.72,45.93){\line(1,0){.974}}
\put(33.67,45.93){\line(1,0){.974}}
\put(35.62,45.93){\line(1,0){.974}}
\put(37.57,45.93){\line(1,0){.974}}
\put(39.52,45.93){\line(1,0){.974}}
\put(41.47,45.93){\line(1,0){.974}}
\put(43.42,45.93){\line(1,0){.974}}
\put(45.37,45.93){\line(1,0){.974}}
\put(47.31,45.93){\line(1,0){.974}}
\put(49.26,45.93){\line(1,0){.974}}
\put(51.21,45.93){\line(1,0){.974}}
\put(53.16,45.93){\line(1,0){.974}}
\put(55.11,45.93){\line(1,0){.974}}
\put(57.06,45.93){\line(1,0){.974}}
\put(59.01,45.93){\line(1,0){.974}}
\put(60.96,45.93){\line(1,0){.974}}
\put(0,0){}
\thicklines
\put(62,46){\line(0,1){3}}
\thinlines
\put(24,44){\line(1,0){.07}}
\put(6.93,43.93){\line(1,0){.944}}
\put(8.82,43.93){\line(1,0){.944}}
\put(10.71,43.93){\line(1,0){.944}}
\put(12.6,43.93){\line(1,0){.944}}
\put(14.49,43.93){\line(1,0){.944}}
\put(16.37,43.93){\line(1,0){.944}}
\put(18.26,43.93){\line(1,0){.944}}
\put(20.15,43.93){\line(1,0){.944}}
\put(22.04,43.93){\line(1,0){.944}}
\put(0,0){}
\put(20,50){\line(1,0){.07}}
\put(9.93,49.93){\line(1,0){.909}}
\put(11.75,49.93){\line(1,0){.909}}
\put(13.57,49.93){\line(1,0){.909}}
\put(15.38,49.93){\line(1,0){.909}}
\put(17.2,49.93){\line(1,0){.909}}
\put(19.02,49.93){\line(1,0){.909}}
\put(0,0){}
\thicklines
\put(20,50){\line(0,1){3}}
\thinlines
\put(22,52){\line(1,0){.07}}
\put(19.93,51.93){\line(1,0){.667}}
\put(21.26,51.93){\line(1,0){.667}}
\put(0,0){}
\thicklines
\put(22,52){\line(0,1){2}}
\thinlines
\put(4,40){\vector(0,1){30}}
\put(2.941,39.941){\line(1,0){.9866}}
\put(4.915,39.941){\line(1,0){.9866}}
\put(6.888,39.941){\line(1,0){.9866}}
\put(8.861,39.941){\line(1,0){.9866}}
\put(10.834,39.941){\line(1,0){.9866}}
\put(12.808,39.941){\line(1,0){.9866}}
\put(14.781,39.941){\line(1,0){.9866}}
\put(16.754,39.941){\line(1,0){.9866}}
\put(18.727,39.941){\line(1,0){.9866}}
\put(20.701,39.941){\line(1,0){.9866}}
\put(22.674,39.941){\line(1,0){.9866}}
\put(24.647,39.941){\line(1,0){.9866}}
\put(26.62,39.941){\line(1,0){.9866}}
\put(28.593,39.941){\line(1,0){.9866}}
\put(30.567,39.941){\line(1,0){.9866}}
\put(32.54,39.941){\line(1,0){.9866}}
\put(34.513,39.941){\line(1,0){.9866}}
\put(36.486,39.941){\line(1,0){.9866}}
\put(38.46,39.941){\line(1,0){.9866}}
\put(40.433,39.941){\line(1,0){.9866}}
\put(42.406,39.941){\line(1,0){.9866}}
\put(44.379,39.941){\line(1,0){.9866}}
\put(46.353,39.941){\line(1,0){.9866}}
\put(48.326,39.941){\line(1,0){.9866}}
\put(50.299,39.941){\line(1,0){.9866}}
\put(52.272,39.941){\line(1,0){.9866}}
\put(54.245,39.941){\line(1,0){.9866}}
\put(56.219,39.941){\line(1,0){.9866}}
\put(58.192,39.941){\line(1,0){.9866}}
\put(60.165,39.941){\line(1,0){.9866}}
\put(62.138,39.941){\line(1,0){.9866}}
\put(64.112,39.941){\line(1,0){.9866}}
\put(66.085,39.941){\line(1,0){.9866}}
\thicklines
\put(53,61){\line(0,1){2}}
\thinlines
\put(50,63){\line(1,0){.07}}
\put(47.93,62.93){\line(1,0){.667}}
\put(49.26,62.93){\line(1,0){.667}}
\put(0,0){}
\thicklines
\put(50,63){\line(0,1){3}}
\thinlines
\put(52,64){\line(1,0){.07}}
\put(49.93,63.93){\line(1,0){.667}}
\put(51.26,63.93){\line(1,0){.667}}
\put(0,0){}
\thicklines
\put(52,64){\line(0,1){3}}

\put(4,17){\line(1,-1){3}}
\multiput(19,18)(.02777778,-.02777778){36}{\line(0,-1){.02777778}}
\put(20,19){\line(1,-1){10}}
\multiput(30,15)(.02777778,-.02777778){72}{\line(0,-1){.02777778}}
\put(53,14){\line(1,-1){9}}
\put(7,68.875){\makebox(0,0)[cc]{a)}}
\put(7,33.75){\makebox(0,0)[cc]{b)}}
\thinlines
\put(3,56){\line(1,0){62}}
\put(1.375,56.5){\makebox(0,0)[cc]{$t$}}
\multiput(7,21)(.02777778,-.02777778){36}{\line(0,-1){.02777778}}
\put(8,21){\line(1,-1){3}}
\multiput(11,21)(.02777778,-.02777778){72}{\line(0,-1){.02777778}}
\put(13,21){\line(1,-1){6}}
\put(32,21){\line(1,-1){5}}
\multiput(37,21)(.02777778,-.02777778){36}{\line(0,-1){.02777778}}
\put(38,21){\line(1,-1){4}}
\multiput(42,21)(.02777778,-.02777778){36}{\line(0,-1){.02777778}}
\put(43,21){\line(1,-1){10}}
\put(3,5){\vector(1,0){62}}
\put(4,5){\vector(0,1){30}}
\put(3,21){\line(1,0){55}}
\put(1.5,21.5){\makebox(0,0)[cc]{$t$}}
\end{picture}
\caption{a) A splitting tree; b) The jumping chronological contour process associated with the same splitting tree after truncation at time $t$. Here $N_t=9$.}
\label{fig : jccp}
\end{figure}
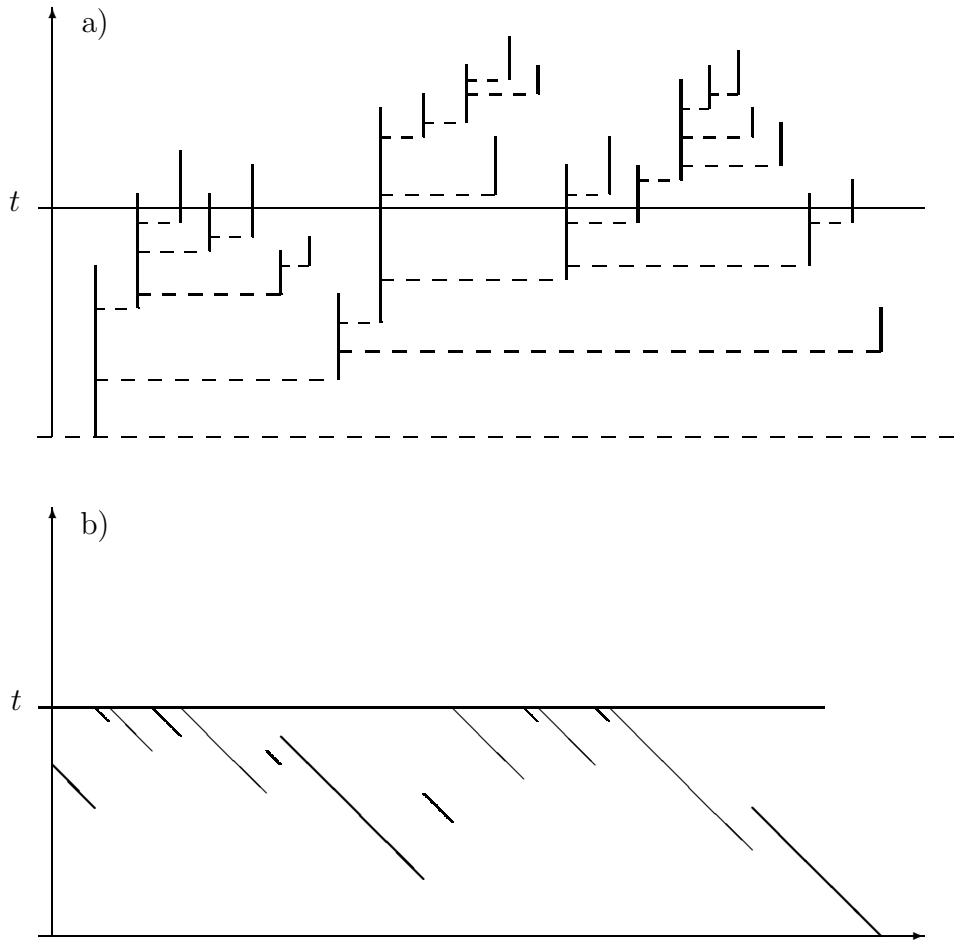

He has shown \cite[Theorem 4.3]{L} that the JCCP is a Markov process, more specifically, it has the same law as the compound Poisson process $X$ with jump measure $\pi$, compensated at rate $-1$, reflected below $t$, and killed upon hitting 0. 

We denote the law of $X$ by $P$, to make the difference with the law $\PP$ of the CMJ process. As seen previously, we record the lifetime duration, say $s$, of the progenitor, by writing $P_s$ for its conditional law on $X_0=s$. To stick to the analogous notation for $\PP$, it will be implicit in the absence of subscript that $X_0$ under $P$ has probability distribution $\mu$.

Let us be a little more specific about the JCCP. Recall that this process visits all existence times of all individuals of the truncated tree. For any individual of the tree, we denote by $\alpha$ its birth time and by $\omega$ its death time. When the visit of an individual $v$ with lifespan $(\alpha(v), \omega(v)]$ begins, the value of the JCCP is $\omega(v)$. The JCCP then visits all the existence times of $v$'s lifespan at constant speed $-1$. If $v$ has no child, then this visit lasts exactly the lifespan of $v$; if $v$ has at least one child, then the visit is interrupted each time a birth time of one of $v$'s daughters, say $w$, is encountered (youngest child first since the visit started at the death level). At this point, the JCCP jumps from $\alpha(w)$ to $\omega(w)\wedge t$  and starts the visit of the existence times of $w$. Since a  truncated tree has finite length, the visit of $v$ has to terminate: it does so at the chronological level $\alpha(v)$ and continues the exploration of the existence times of $v$'s mother, at the height (time) where it had been interrupted.
This procedure then goes on recursively and terminates as soon as $0$ is encountered (birth time of the progenitor). See Figure \ref{fig : jccp} for an example. 

Note that the genealogy of a  splitting tree truncated at time $t$ can be coded by associating each individual with a word of integers, such that $\varnothing$ is the root, 1 is the last daughter of the root born before time $t$, 2 is the next before last daughter of the root,..., 11 is the last daughter of 1 (born before time $t$), and so on. Then the order in which individuals get their first visit by the contour process is the lexicographical order associated with this (so-called Ulam--Harris--Neveu) labelling. Roughly speaking, if $u$ and $v$ are two distinct 
finite words of integers, and that $h_u$ (resp. $h_v$) is the first integer in the word $u$ (resp. $v$) coming immediately after their longest common prefix, then 
$u$ comes first in the lexicographical order if and only if $h_u<h_v$. Here we assume that if $h \neq \varnothing$, then  $h > \varnothing$. 

Since the JCCP is Markovian (as seen earlier, it is a reflected, killed L{\'e}vy process), its excursions between consecutive visits of points at height $t$ are i.i.d.\ excursions of $X$ away from $(t,+\infty]$. Observe in particular that the number of visits of $t$ by $X$ is exactly the number $N_t$ of individuals alive at time $t$.
Therefore, it is easy to see that $N_t$ has a shifted geometric distribution with parameters specified as follows. Let $\tau_A$ denote the first hitting time of the set $A$ by $X$. We also use the following shortcuts 
$$
\tau_{x}:=\tau_{\{x\}}\quad \mbox{ and }\quad \tau_{x}^+:=\tau_{(x,+\infty]},
$$
that is, $\tau_x$ is the first hitting time of $x$ and $\tau_x^+$ is the first hitting time of the open interval
$(x,+\infty]$.
Then conditional on the initial progenitor to have lived $s$ units of time, we have
\begin{equation}
\label{eqn : N_t=0}
\PP_s(N_t=0)=P_s(\tau_0< \tau^+_t),
\end{equation}
and, applying recursively the strong Markov property,
\begin{equation}
\label{eqn : N_t=k condition}
\PP_s(N_t=n \mid N_t\not=0) = P_t(\tau^+_t<\tau_0)^{n-1}P_t(\tau_0< \tau^+_t).
\end{equation}
Note that the subscript $s$ in the last display is useless. Also note that the spatial 
homogeneity of L{\'e}vy processes implies $P_t(\tau_0< \tau^+_t)=P_0(\tau_{-t}< \tau^+_0)$.

In addition, exact formulae can be deduced for \eqref{eqn : N_t=0} and \eqref{eqn : N_t=k condition} from the fact that the JCCP is a L{\'e}vy process with no negative jumps, using scale functions of the L{\'e}vy process $X$. This part is developed in the final section of the paper. 

Later, we see that that the population size is not only (conditionally) geometric at fixed times, but also at the first detection time $T$, using the same decomposition of the contour process into excursions away from $(T,+\infty]$. This decomposition is displayed in Subsection 3.1 of Section 3  `Main results'.  
Subsections 3.2 and 3.3 are devoted to path decompositions providing equalities in law for the whole contour process of the tree stopped at $T$, involving in particular Vervaat's transformation: see Theorems \ref{thm1}, \ref{thm2} and especially \ref{thm3} for the result stated in the abstract. Section 4 focuses on the joint distribution of $T$, $N_T$ and of the ages and residual lifetimes of the $N_T$ alive individuals at time $T$. Subsections 4.2 and 4.3 provide explicit formulae (up to scale functions of the L{\'e}vy process $X$) for these distributions. The reader interested in applications can go straightforward to the last statement of Section 4, Proposition \ref{last}. Finally, Section 5 extends these results to the example of a pathogen outbreak in the hospital modeled by a Crump--Mode--Jagers process with constant transmission rate $b$ and i.i.d.\ (infection) lifetimes, but also taking into account the length of stay in the hospital up to infection.

In the rest of the paper, we always use the following notation, where $E$ can be any expectation operator, $A,B$ any events and $Z$ any (positive or integrable) random variable $E(Z,A,B) := E(Z \mathbbm{1}_{A\cap B}).$

\subsection{Intuition for the geometric distribution}
In this section, we show how to gain insight from the equivalence of the splitting tree and the corresponding contour process, as visualised in Figure \ref{fig : jccp}, and in particular to give the intuition why the number of individuals at the first detection time is geometrically distributed \cite{K,T}.  This intuition also gives the main ideas behind the rigorous proofs below. 


We consider the event that $N_t = n$ ($n>0$) and no detection has occurred yet at time $t$, i.e., the event $\{N_t=n,T>t\}$.
This event occurs if the following events successively occur 
\begin{enumerate}
\item
 The L{\'e}vy process $X$ following the contour of the tree truncated below time $t$ starts with a typical jump (distributed according to $\mu$) and hits the interval $(t,\infty]$ before it hits 0 again, and during this time no clock rings. The probability of this event is  $ E(e^{-\delta \tau^+_t}, \tau_t^+ < \tau_0)$;
\item 
the process $X$ started at $t$, makes an excursion ending in the interval $(t,\infty]$ without hitting 0, and no clock rings during this excursion. Since the contour process of the tree truncated below $t$ is started again at $t$, independently from the past, $n-1$  such events occur successively, and each of them occurs independently with probability $E_t(e^{-\delta \tau^+_t},\tau_t^+ < \tau_0)$;
\item 
$X$ starts at $t$ and reaches $0$ before hitting the interval $(t,\infty]$, and during this time no clock rings. This happens with probability $E_t(e^{-\delta \tau_0},\tau_t^+ > \tau_0)$.
\end{enumerate}

The next step is to ``glue'' the path described in the third event above before the path  described in the first event above, into one excursion with infimum equal to 0 and in which no clock rings. This is basically the inverse of Vervaat's transformation, see Figure \ref{fig : vervaat} below, where the inverse is constructed in such a way that the infimum of the whole process is performed during this newly created (first) excursion.

Since $X$ jumps at rate $b$ and has slope $-1$ (and also is translation invariant), multiplying the probability of this concatenated path by $b \ dt$ gives the probability of an excursion of $X$ away from $[0,+\infty)$ without ringing clock and with infimum in $(-t-dt,-t)$. Then,  $b\ \PP(N_t=n,T>t)\ dt$ is the probability that $X$ makes $n$ excursions away from $[0,+\infty)$ without making a clock ring, and that the infimum of the whole path is performed in the first excursion and belongs to $(-t-dt, -t)$. This yields
\begin{multline}
b\ \mathbb{P}(N_t=n, T>t)\ dt = E_0\big(e^{-\delta \tau^+_0},-\inf_{0<s<\tau^+_{0}} X_s \in dt \big) \left(E_0(e^{-\delta \tau^+_0},\tau^+_0<\tau_{-t}) \right)^{n-1}\\
	=E_0\big(e^{-\delta \tau^+_0},-\inf_{0<s<\tau^+_{0}} X_s \in dt \big) \left(E_0\big(e^{-\delta \tau^+_0},-\inf_{0<s<\tau^+_{0}} X_s<t\big) \right)^{n-1}\\
= \frac{1}{n}\frac{d}{dt}\left(E_0\big(e^{-\delta \tau^+_0},-\inf_{0<s<\tau^+_{0}} X_s<t\big) \right)^{n}\ dt.
\end{multline}
Now observe that $$\mathbb{P}(N_t=n, T \in dt) = \delta n\ \mathbb{P}(N_t=n, T>t)\ dt = 
\frac{\delta}{b} \frac{d}{dt}\left(E_0(e^{-\delta \tau^+_0},\tau^+_0<\tau_{-t}) \right)^{n} dt.$$
Integrating over $t$ now gives $$\mathbb{P}(N_T=n, T<\infty) = \frac{\delta}{b}  \left(E_0(e^{-\delta \tau^+_0}) \right)^{n}.$$ 

A little elaboration on this argument also gives that the distributions of the ages and residual lifetimes at time $T$ should be i.i.d.\ and independent of $N_T$. Precise proofs are given in the sections below.

\section{Main general results}

\subsection{Decomposition of the splitting tree at first detection time}
In this subsection, we call any \cadlag\ (continue {\`a} droite, avec des limites {\`a} gauche, i.e., right-continuous with left limits \cite[p.\ 346]{Grim92}) path $\epsilon$ with lifetime $V(\epsilon)\in [0,+\infty]$, \emph{an excursion}. We use the notation $\mathscr E$ for the space of excursions, endowed with Skorokhod's topology and the associated Borel $\sigma$-field.\\

For any time $t>0$, we set $\rho_t$
the first exit time of $(0,t)$ 
and we let $w_0^t$ denote the finite path of the JCCP $X^{(t)}$ killed upon exiting $(0,t)$, that is, $w_0^t:=(X^{(t)}_s;s\le \rho_t)$.
Further, on the event $N_t=n\ge 1$, for $i=1,\ldots, n$, we set $\sigma_i$ the $i$-th hitting time of $t$ by $X^{(t)}$  and we let $w_i^t$ denote the path of the JCCP $X^{(t)}$ between times $\sigma_i$ and $\sigma_{i+1}$, with the convention that $\sigma_{n+1}=\tau_0$, that is,
$$
w_i^t(s):= X^{(t)}_{s+\sigma_i} \qquad s< \sigma_{i+1}- \sigma_{i}.
$$
For $i=0,\ldots,n$, we denote by $\ell_i:= V(w_i^t)=\sigma_{i+1}- \sigma_{i}$ the lifetime of the excursion $w_i^t$ (so that $w_n(\ell_n)=0$), and for $i=0,\ldots,n-1$, we record the size of the jump made by the contour process before reflection by setting $w_i^t(\ell_i)$ equal to the date of death of the individual alive at $t$ visited at time $\sigma_{i+1}$. In particular, $\ell_0$ is the life length of the progenitor, and if $\ell_0>t$, then $w_0^t$ is reduced to the one-point process that maps 0 to $\ell_0$.

In particular, when $t$ is fixed, we know \cite{L} that conditional on $\{N_t=n\}$, 
\begin{itemize}
\item
the excursion $w_0^t$ follows the law 
 of $X$ started at a jump distributed as $\mu$, killed at $\tau_t^+$ and conditioned on $\tau_t^+<\tau_0$;
\item
the excursions $w_i^t$, $i=1,\ldots, n-1$, are i.i.d.\ and follow the law 
of $X$ started at $t$, killed at $\tau_t^+$ and conditioned on $\tau_t^+<\tau_0$;
\item
the excursion $w_n^t$ follows  the law 
of $X$ started at $t$, killed at $\tau_0$ and conditioned on $\tau_0<\tau_t^+$.
\end{itemize}
Now recall that all individuals are equipped with an independent exponential clock with parameter $\delta$, and that the time when the first of those clocks rings is denoted by $T$ and called detection time.

\begin{prop}
\label{prop1}
Let $p:=\PP(T<\infty)$ be the probability that at least one clock rings before extinction of the population. Then
$$
p=E(1-e^{-\delta \tau_0}).
$$
More specifically, 
$$
\PP(N_t= 0, T>t) = \PP(N_t= 0, T=\infty) = E \big(e^{-\delta \tau_0}, \tau_0<\tau_t^+\big)
$$
and for any $n\ge 1$,
$$
\PP(N_t= n, T>t)=
E \big(e^{-\delta \tau_t^+},\tau_t^+ <\tau_0\big)\, \left(E_t\big(e^{-\delta \tau_t^+}, \tau_t^+ <\tau_0\big)\right)^{n-1} \,E_t\big(e^{-\delta \tau_0}, \tau_0 <\tau_t^+\big).
$$
Further, let $n\ge 1$ and $G', G,  F_0, F_1,\ldots, F_{n-1}$ be non-negative, measurable functions on $\mathscr E$. Then
\begin{multline*}
\EE\left(G'(w_t^0) G(w_t^n)\prod_{i=1}^{n-1} F_i(w_t^i),N_T=n, T\in dt\right) =\\ \delta\, n\, dt \,
E\left(G'(X_s;s\le \tau_t^+)\,e^{-\delta \tau_t^+}, \tau_t^+ <\tau_0\right)\times\\
	\times \left\{\prod_{i=1}^{n-1} E_t(F_i(X_s;s\le \tau_t^+)\,e^{-\delta \tau_t^+}, \tau_t^+ <\tau_0)\right\}
	E_t\left(G(X_s;s\le \tau_0)\,e^{-\delta \tau_0}, \tau_0<\tau_t^+\right).
\end{multline*}
\end{prop}

\begin{proof}
In the whole proof, let $\mathbf{e}$ denote an independent exponential r.v.\ with parameter $\delta$.
Since the JCCP has slope $-1$ and jump sizes equal to life lengths, the lifetime $\tau_0$ of the contour process is exactly the sum of the lifespans of all individuals in the population. As a consequence, 
$$
p=P(\tau_0 >\mathbf{e})= E(1-e^{-\delta \tau_0}).
$$
Applying this property to the truncated contour process $X^{(t)}$ and using the path decomposition preceding the statement of the proposition, we get
\begin{eqnarray*}
\PP(N_t= n, T>t)&=& \PP(N_t= n, \tau_0\big(X^{(t)}\big)<\mathbf{e})=\EE\left( e^{-\delta\tau_0\left(X^{(t)}\right)}, N_t= n\right)\\
	&=&\PP(\tau_t^+ <\tau_0)\,E \big(e^{-\delta V(w_0^t)}\big) \, \left(P_t(\tau_t^+ <\tau_0)\,E \big(e^{-\delta V(w_1^t)}\big) \right)^{n-1} \,\times\\
	&\times&
	P_t( \tau_0 <\tau_t^+)\, E\big(e^{-\delta V(w_n^t)}\big),
\end{eqnarray*}
which yields the desired expression. 

Observing that the  detection rate equals $\delta n$ conditional on $\{N_t=n\}$, 
we finally get
\begin{multline*}
\EE\left(G'(w_t^0) G(w_t^n)\prod_{i=1}^{n-1} F_i(w_t^i),N_T=n, T\in dt\right) =\\ \delta\, n\, dt \,\EE\left(G'(w_t^0) G(w_t^n)\prod_{i=1}^{n-1} F_i(w_t^i),N_t=n, T>t\right) ,
\end{multline*}
and the desired equality follows by the same method as previously.
\end{proof}

\begin{rem}
Since the knowledge of the contour of the genealogical tree yields that of the tree itself, the previous proposition characterises the law of the splitting tree stopped at the first detection time (noting that conditional on $N_T=n$, the marked individual is of course uniform among all $n$ alive individuals).
\end{rem}

\subsection{Rephrasing with i.i.d.\ excursions}
In this subsection, $\epsilon$ denotes an excursion \textbf{distributed as} $\mathbf{X}$ \textbf{started at 0 and killed upon hitting} $(0,+\infty]$. Recall that $\epsilon$ only takes negative values, except at $0$ ($\epsilon(0)=0$) and at $V$, since $\epsilon(V)>0$ on the event $\{V<\infty\}$ (on the complementary event, $\eps$ drifts to $-\infty$). 

Set $\jmath(\epsilon):=\inf_s \epsilon(s)$. On the event $\{\jmath\not=-\infty\}$ (which coincides a.s. with $\{V<\infty\}$), we denote by $h(\epsilon)$ the unique time $h$ such that $\epsilon(h-)=\jmath$. 
Also,
we denote by $\epsilon^\leftarrow$ the pre-$h$ process and by $\epsilon^\rightarrow$ the post-$h$ process:
$$
\epsilon^\leftarrow(s):= \epsilon(s)\qquad 0\le s<h(\epsilon),
$$
with $\epsilon^\leftarrow(h(\epsilon)) = \epsilon^\leftarrow(h(\epsilon)-)=\jmath(\epsilon)$ and
$$
\epsilon^\rightarrow(s):= \epsilon(s+h(\epsilon))\qquad 0\le s\le V(\epsilon)-h(\epsilon).
$$
Notice that with positive probability $V(\epsilon)=h(\epsilon)$, so that $\epsilon^\rightarrow$ is then reduced to the one-point process that maps 0 to $ \epsilon(V)$.

Let $n\ge 1$ and $\epsilon_1,\ldots, \epsilon_n$ denote i.i.d.\  excursions  distributed as $\epsilon$. 
Set 
$$
I_n:=\min_k \jmath (\epsilon_k)\quad\mbox{ and }\quad K_n:=\mbox{arg}\min_k \jmath (\epsilon_k).
$$
The next result is a consequence of the  following two lemmas and Proposition \ref{prop1}.
\begin{thm}
\label{thm1}
Let $G, G',F_1,\ldots, F_{n}$ be non-negative, measurable functions on $\mathscr E$. Then 
\begin{multline*}
\EE\left(G'(w_t^0) G(w_t^n)\prod_{i=1}^{n-1} F_i(w_t^i),N_T=n, T\in dt\right) =\\ \frac{\delta}{b} \ .\ 
E\left( G(\epsilon_{K_n}^\leftarrow+t)\,G'(\epsilon_{K_n}^\rightarrow+t)\,\prod_{k\not=K_n} F_{k-K_n \ mod(n)}(\epsilon_k+t)\,\prod_{k=1}^n e^{-\delta V(\epsilon_k)}, -I_n \in dt\right).
\end{multline*}

\end{thm}
\begin{rem}
Note that the expression inside the expectation in the rhs has zero probability when one of the excursions has infinite lifetime, that is, when there is some $k$ such that $V(\eps_k)=+\infty$.
\end{rem}
\begin{lem}
Let $G$ and  $G'$ be two non-negative, measurable functions on $\mathscr E$. Then
$$
E( G(\epsilon^\leftarrow)\,G'(\epsilon^\rightarrow), -\jmath \in dt) =
b\, dt\,E_0(G(X_s;s\le \tau_{-t}), \tau_{-t}<\tau_0^+)\,E(G'(X_s-t;s\le \tau_t^+),\tau_t^+<\tau_0).
$$

\end{lem}

\begin{proof} Applying the strong Markov property at $\tau_{-t}$ yields
\begin{eqnarray*}
E( G(\epsilon^\leftarrow)\,G'(\epsilon^\rightarrow),-\jmath \in dt)
	&=& E( G(\epsilon^\leftarrow)\,G'(\epsilon^\rightarrow), -\jmath \in dt, \tau_{-t}<\tau_0^+)\\
	&=& E_0(G(X_s;s\le \tau_{-t}), \tau_{-t}<\tau_0^+) \times\\
	&\times&\,dt\,\int_{(0,+\infty]}\pi(dy)\,E_{y-t}(G'(X_s;s\le \tau_0^+),\tau_0^+<\tau_{-t}),
\end{eqnarray*}
which yields the result.
\end{proof}

\begin{lem}
Let $G, G',F_1,\ldots, F_{n}$ be non-negative, measurable functions on $\mathscr E$. Then for $j=1,\ldots,n$, 
\begin{multline}
E\big( G(\epsilon_{K_n}^\leftarrow)\,G'(\epsilon_{K_n}^\rightarrow)\prod_{k\not={K_n}} F_k(\epsilon_k), -I_n \in dt, K_n=j\big) =\\
b\, dt\,E_0(G(X_s;s\le \tau_{-t}), \tau_{-t}<\tau_0^+)\,E(G'(X_s-t;s\le \tau_t^+),\tau_t^+<\tau_0)\,\prod_{k\not=j}E_0(F_k(X_s;s\le \tau_{0}^+), \tau^+_{0}<\tau_{-t}).
\end{multline}

\end{lem}

\begin{proof} The expression in the lhs equals
\begin{eqnarray*}
	& &E\big( G(\epsilon_{K_n}^\leftarrow)\,G'(\epsilon_{K_n}^\rightarrow)\prod_{k\not={K_n}} F_k(\epsilon_k), -I_n \in dt, K_n=j \big)\\
&=& 		E\big( G(\epsilon_j^\leftarrow)\,G'(\epsilon_j^\rightarrow), -\jmath(\eps_j) \in dt\big)E\big(\prod_{k\not=j} F_k(\epsilon_k), -\jmath(\eps_k) <t\ \forall k\not=j\big) \\
	&=& E\big( G(\epsilon_j^\leftarrow)\,G'(\epsilon_j^\rightarrow), -\jmath(\eps_j) \in dt\big)\prod_{k\not=j}E_0(F_k(X_s;s\le \tau_{0}^+), \tau^+_{0}<\tau_{-t}),
\end{eqnarray*}
and the conclusion stems from the previous lemma.
\end{proof}

\subsection{Rephrasing with Vervaat's  transformation}
Forgetting about the terminal jump of each excursion (piece of information that actually is useful in the next section), Theorem \ref{thm1} can be expressed in a more elegant way. 

For any \cadlag\ path $Z$ with finite lifetime $V(Z)$ and law locally absolutely continuous w.r.t.\ $X$, we set $I(Z):=\inf Z$ and we define $H(Z)$ as the unique time $t$ such that $Z(t-) = I(Z)$. Finally we let $Z'$ denote Vervaat's transform of $Z$, defined as the path with lifetime  $V(Z)$ such that $Z'(V(Z)) = 0$ and
$$
Z'(s) = Z(s+H(Z) \ mod(V(Z))) - I(Z)\qquad 0\le s < V(Z).
$$

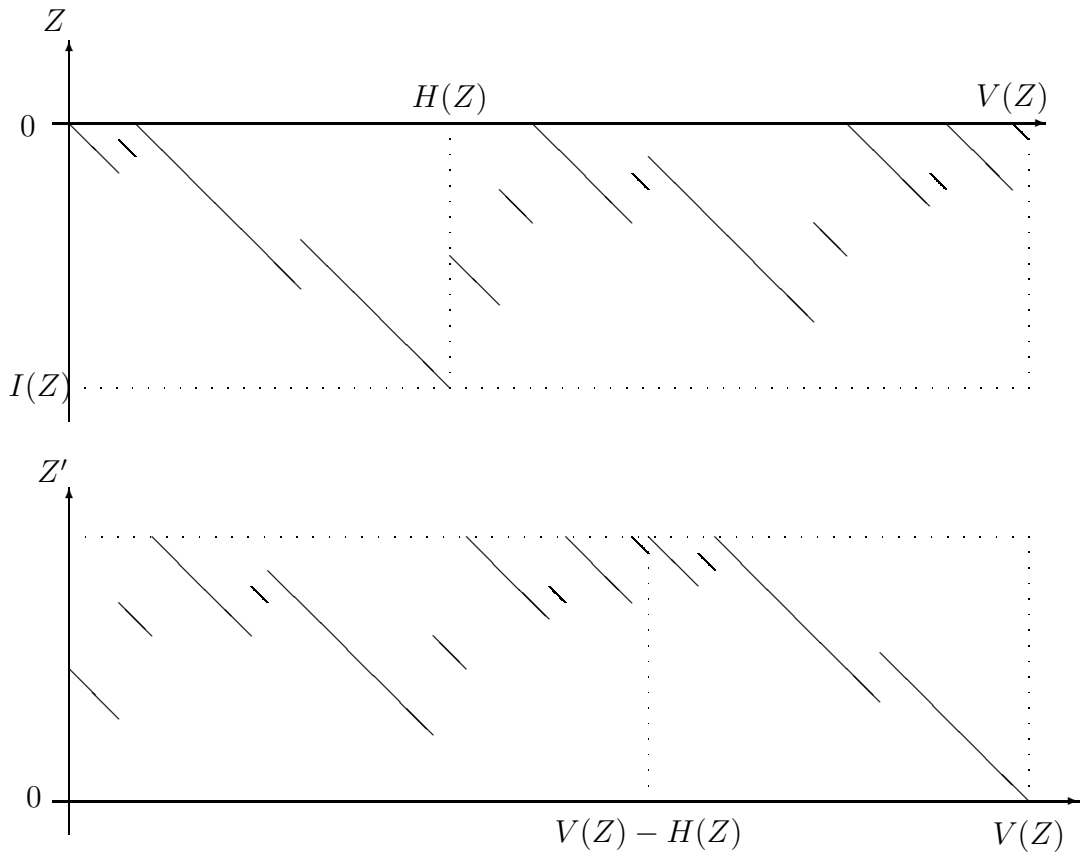
\begin{figure}[ht]
\unitlength 2.2mm 
\linethickness{0.2pt}
\ifx\plotpoint\undefined\newsavebox{\plotpoint}\fi 
\begin{picture}(65,52.25)(0,0)
\put(0,25){}
\put(0,25){}
\put(0,25){}
\put(0,0){}
\put(0,0){}
\put(0,0){}
\put(0,0){}
\put(0,0){}
\put(0,0){}
\put(0,0){}
\put(0,0){}
\put(0,0){}
\put(0,0){}
\put(0,0){}
\put(0,0){}
\put(0,0){}
\put(0,0){}
\put(0,0){}
\put(0,0){}
\put(0,0){}
\put(0,0){}
\put(0,0){}
\put(0,0){}
\put(0,0){}
\thinlines
\put(27,38){\line(1,-1){3}}
\put(4,13){\line(1,-1){3}}
\multiput(38,43)(.008849558,-.008849558){113}{\line(0,-1){.008849558}}
\multiput(15,18)(.008849558,-.008849558){113}{\line(0,-1){.008849558}}
\put(39,44){\line(1,-1){10}}
\put(16,19){\line(1,-1){10}}
\put(49,40){\line(1,-1){2}}
\put(26,15){\line(1,-1){2}}
\put(30,42){\line(1,-1){2}}
\put(7,17){\line(1,-1){2}}
\put(18,39){\line(1,-1){9}}
\put(53,14){\line(1,-1){9}}
\thinlines
\put(4,46){\line(1,-1){3}}
\put(39,21){\line(1,-1){3}}
\put(32,46){\line(1,-1){6}}
\put(9,21){\line(1,-1){6}}
\put(51,46){\line(1,-1){5}}
\put(28,21){\line(1,-1){5}}
\multiput(56,43)(.008849558,-.008849558){113}{\line(0,-1){.008849558}}
\multiput(33,18)(.008849558,-.008849558){113}{\line(0,-1){.008849558}}
\put(57,46){\line(1,-1){4}}
\put(34,21){\line(1,-1){4}}
\multiput(61,46)(.008849558,-.008849558){113}{\line(0,-1){.008849558}}
\multiput(38,21)(.008849558,-.008849558){113}{\line(0,-1){.008849558}}
\put(8,46){\line(1,-1){10}}
\put(43,21){\line(1,-1){10}}
\put(3,5){\vector(1,0){62}}
\put(1.5,45.875){\makebox(0,0)[cc]{$0$}}
\multiput(7,45)(.008849558,-.008849558){113}{\line(0,-1){.008849558}}
\multiput(42,20)(.008849558,-.008849558){113}{\line(0,-1){.008849558}}
\put(3,46){\vector(1,0){60}}
\put(4,28){\vector(0,1){23}}
\put(4,3){\vector(0,1){21}}
\put(27,47.5){\makebox(0,0)[cc]{$H(Z)$}}
\put(2.25,30){\makebox(0,0)[cc]{$I(Z)$}}
\put(3.125,52.25){\makebox(0,0)[cc]{$Z$}}
\multiput(26.982,45.982)(0,-.94118){18}{{\rule{.2pt}{.2pt}}}
\multiput(26.982,29.982)(-.95833,0){25}{{\rule{.2pt}{.2pt}}}
\multiput(61.982,45.982)(0,-.94118){18}{{\rule{.2pt}{.2pt}}}
\multiput(61.982,29.982)(-.972222,0){37}{{\rule{.2pt}{.2pt}}}
\multiput(38.982,4.982)(0,.94118){18}{{\rule{.2pt}{.2pt}}}
\multiput(61.982,4.982)(0,.94118){18}{{\rule{.2pt}{.2pt}}}
\multiput(61.982,20.982)(-.983051,0){60}{{\rule{.2pt}{.2pt}}}
\put(3,25){\makebox(0,0)[cc]{$Z'$}}
\put(1.875,5.25){\makebox(0,0)[cc]{$0$}}
\put(61,47.5){\makebox(0,0)[cc]{$V(Z)$}}
\put(62,2.875){\makebox(0,0)[cc]{$V(Z)$}}
\put(39,3){\makebox(0,0)[cc]{$V(Z)-H(Z)$}}
\end{picture}
\caption{Vervaat's transformation. Top panel: A path $Z$ with finite lifetime $V(Z)$ performing its infimum $I(Z)$ at time $H(Z)-$; Bottom panel: Vervaat's transform $Z'$ of $Z$ obtained by shifting $Z$ by $-I(Z)$ and performing a circular time-change starting at time $H(Z)$.}
\label{fig : vervaat}
\end{figure}

More specifically,
$$
Z'(s):=\left\{
\begin{array}{lcr}
Z(s+H(Z)) - I(Z)&\mbox{ if }&0\le s<V(Z)-H(Z)\\
Z(s+H(Z)-V(Z))-I(Z)&\mbox{ if }& V(Z)-H(Z)\le s <V(Z)\\
Z'(s) = 0&\mbox{ if }&  s =V(Z).
\end{array}
\right.
$$
Note that $Z'$ takes positive values, apart from its terminal value equal to 0 (and that $Z'$ is left-continuous at this point).\\
\\
Now let $Y_n$ denote the concatenation of the $n$ i.i.d.\ excursions $(\eps_i)_{i=1,\ldots,n}$ of the last subsection. In particular, $Y_n$ is equally distributed as the L{\'e}vy process $X$ reflected below its supremum and killed at its $(n+1)$-st
hitting time of $0$. Then observe that $I_n=I(Y_n)$ and let $Y_n'$ denote Vervaat's transformation of $Y_n$. We have the following corollary of Theorem \ref{thm1}.
\begin{thm}
\label{thm2}
For any $n\ge 1$,
$$
\PP\left(N_T=n, T\in dt, X^{(T)}\in d\varepsilon\right) = \frac{\delta}{b} \ e^{-\delta V(\varepsilon)}\
P\left( -I_n \in dt, Y_n'\in d\varepsilon\right).
$$
\end{thm}
\begin{proof}
From Theorem \ref{thm1} we get
\begin{multline*}
\EE\left(G'(w_t^0) G(w_t^n)\prod_{i=1}^{n-1} F_i(w_t^i),N_T=n, T\in dt\right) =\\ \frac{\delta}{b} \ .\ 
E\left( G(\epsilon_{K_n}^\leftarrow-I_n)\,G'(\epsilon_{K_n}^\rightarrow-I_n)\,\prod_{k\not=K_n} F_{k-K_n \ mod(n)}(\epsilon_k-I_n)\,e^{-\delta V(Y_n')}, -I_n \in dt\right),
\end{multline*}
which, by a monotone class theorem \cite[p.2]{Kall}, ensures that for any non-negative measurable function $F$ on $\mathscr E$,
$$
\EE\left(F(X^{(t)}), N_T=n, T\in dt\right) = \frac{\delta}{b}\
E\left(F(Y_n')\,e^{-\delta V(Y_n')}, -I_n \in dt\right),
$$
hence the result.
\end{proof}

One can now push this path decomposition even further by starting from a path, say $Y$, of $X$ reflected below its supremum as well as from an independent exponential random variable $\mathbf{e}$ with parameter $\delta$. Note that $Y$ is the mere concatenation of a sequence $(\eps_i)_{i\ge 1}$ of i.i.d.\ excursions distributed as $\epsilon$ (stopped at the first one with infinite lifetime). Then let $M$ be the unique non-negative integer such that $\mathbf{e}$ falls into the $M+1$-st excursion of $Y$ away from 0
$$
M:= \max\{n\ge 0: V(\eps_1) + \cdots +V(\eps_n)<\mathbf{e}\},
$$
with the usual convention that an empty sum is 0.
As previously, define $Y_M$ as the path $Y$ killed at its $M+1$-st hitting time of 0, set $I_M:= \inf Y_M$ and let $Y_M'$ denote Vervaat's transform of $Y_M$.

\begin{thm}
\label{thm3}
For any $n\ge 1$,$$
\PP\left(N_T=n, T\in dt,X^{(t)}\in d\vareps\right) = \frac{\delta}{b\ E_0\big(1-e^{-\delta \tau_0^+}\big)} \  
P\left(M=n, -I_M \in dt, Y_M'\in d\vareps\right).
$$
\end{thm}

\begin{proof}
Using the definition of $M$, we get
\begin{eqnarray*}
P\left(M=n, -I_M \in dt, Y_M'\in d\vareps\right)	
				&=&   P\left(V(Y_n')<\mathbf{e}<V(Y_n')+V(\epsilon_{n+1}), -I_n \in dt, Y_n'\in d\vareps\right)\\
				&=& e^{-\delta V(\varepsilon)}\ E\big(1-e^{-\delta V(\epsilon_{n+1})}\big)\ P\left( -I_n \in dt, Y_n'\in d\vareps\right)\\
				&=& e^{-\delta V(\varepsilon)}\ E_0\big(1-e^{-\delta \tau_0^+}\big)\ P\left( -I_n \in dt, Y_n'\in d\vareps\right),
\end{eqnarray*}
and an appeal to Theorem \ref{thm2} yields the result.
\end{proof}

\section{Applications and explicit formulae}

\subsection{Some lower dimensional marginals of interest}
\label{subsec:lowerdm}
In this subsection, we give the joint law of $T$ and $N_T$, as well as the joint law of the ages and residual lifetimes $(A_1, R_1, \ldots, A_{N_T}, R_{N_T})$ of the $N_T$ alive individuals at time $T$ on the event $\{T<\infty\}$. 

The next statement follows from Theorem \ref{thm1} by taking all functionals equal to 1. Note that $E_0\left(e^{-\delta \tau_0^+}, \tau_0^+<\tau_{-t}\right)$ can be read as $E_0\left(e^{-\delta \tau_0^+}, -\inf_{0\le s \le \tau_0^+} X_s <t\right)$, so it is differentiable with derivative equal to $E_0\left(e^{-\delta \tau_0^+}, -\inf_{0\le s \le \tau_0^+} X_s \in dt\right)/dt$.
\begin{cor} Let $n\ge 1$ and $y,t>0$.
The joint law of $T$ and $N_T$ is given by
\begin{multline*}
\PP(N_T=n, T\in dt)=\frac{\delta}{b} \ .\
E\left( \prod_{k=1}^n e^{-\delta V(\epsilon_k)}, -I_n \in dt\right) \\= 
\frac{n\delta}{b} \ .\
E_0\left( e^{-\delta \tau_0^+}, -\inf_{0\le s \le \tau_0^+} X_s \in dt \right)\,
\left(E_0\left(e^{-\delta \tau_0^+}, \tau_0^+<\tau_{-t}\right)\right)^{n-1}
.
\end{multline*}
Integrating the variable $t$ over $(0,y)$ yields
$$
\PP(N_T=n, T<y)=\frac{\delta}{b} \ .\
\left(E_0\left(e^{-\delta \tau_0^+}, \tau_0^+<\tau_{-y}\right)\right)^n,
$$
 and finally, letting $y\to\infty$, we get
$$
\PP(N_T=n)=\frac{\delta}{b} \ .\
\left(E_0\left(e^{-\delta \tau_0^+}\right)\right)^n .
$$
\end{cor}

The next statement follows from Theorem \ref{thm1} by reducing the functionals to functions of the bi-variate random variable $(-\eps(V(\eps)-), \eps(V(\eps)))$, known as the undershoot and overshoot of $\epsilon$ at its first up-crossing of the $x$-axis. Indeed, recall that the age $A_i$ and the residual lifetime $R_i$ of the $i$-th individual in the population at time $t$ in the order of the contour,  are seen directly on the JCCP as the undershoot and overshoot of $w_i^t$ across $t$ (which occurs at time $V(w_i^t)$, with terminal value equal to the date of death of this $i$-th individual). We use the notation $Und(\eps):=-\eps(V(\eps)-)$ and $Ove(\eps)=\eps(V(\eps))$ for the undershoot and overshoot of $\eps$.

\begin{cor}
The joint law of the ages and residual lifetimes $(A_1, R_1, \ldots, A_{N_T}, R_{N_T})$ of the $N_T$ alive individuals at time $T$ is given by
\begin{multline*}
\PP(N_T=n, T\in dt, A_i\in da_i,R_i\in dr_i,i=1,\ldots,n)
=\\
\frac{\delta}{b} \ .\ E\left(\prod_{k=1}^n e^{-\delta V(\epsilon_k)} , -I_n \in dt,  Und(\eps_{i-1+K_n\  mod(n)})\in da_i, Ove(\eps_{i-1+K_n\  mod(n)}) \in dr_i, i=1,\ldots, n \right) \\= 
\frac{n\delta}{b} \ .\
E_0\left(e^{-\delta \tau_0^+}, -\inf_{0\le s \le \tau_0^+} X_s \in dt, -X_{\tau_0^+ -}\in da_1, X_{\tau_0^+ }\in dr_1   \right)\times
\\
\times\prod_{k=2}^{n}
E_0\left(e^{-\delta \tau_0^+}, -X_{\tau_0^+ -}\in da_k, X_{\tau_0^+ }\in dr_k,\tau_0^+<\tau_{-t} \right)
.
\end{multline*}
\end{cor}
Recall from Theorem \ref{thm3} (or observe from the last statement) that $A_1$ and $R_1$ are the undershoot and overshoot of the excursion where the infimum is performed. In order to lose this information (which is certainly not in the hands of who observes the beginning of the epidemics), we reshuffle the labels of the individuals at $T$, on the event $\{T<\infty, N_T=n\}$, by drawing independently a uniform permutation $\varsigma$ on $\{1,\ldots,n\}$ and setting
$$
(A_i', R_i'):=(A_{\varsigma(i)}, R_{\varsigma(i)})\qquad i=1,\ldots,n.
$$
The first equality in the next statement is a mere reformulation of Corollary 4.2 using the previous definition. The integration part comes from the same argument as the one mentioned before Corollary 4.1, i.e., by writing the event $\{\tau_0^+<\tau_{-t}\}$ in the form 
$\{-\inf_{0\le s \le \tau_0^+} X_s <t\}$.
\begin{cor}

The joint law of the (reshuffled) ages and residual lifetimes $(A_1', R_1', \ldots, A_{N_T}', R_{N_T}')$ of the $N_T$ alive individuals at time $T$ is given by
\begin{multline*}
\PP(N_T=n, T\in dt, A_i'\in da_i,R_i'\in dr_i,i=1,\ldots,n)
= \\
\frac{\delta}{b} \ .\
\sum_{i=1}^n
E_0\left(e^{-\delta \tau_0^+}, -\inf_{0\le s \le \tau_0^+} X_s \in dt, -X_{\tau_0^+ -}\in da_i, X_{\tau_0^+ }\in dr_i  \right)\times
\\
\times\prod_{ k\not=i}
E_0\left(e^{-\delta \tau_0^+}, -X_{\tau_0^+ -}\in da_k, X_{\tau_0^+ }\in dr_k,\tau_0^+<\tau_{-t}\right)
.
\end{multline*}
Integrating the variable $t$ over $(0,y)$ yields
\begin{multline*}
\PP(N_T=n, T<y, A_i'\in da_i,R_i'\in dr_i,i=1,\ldots,n)=\\
\frac{\delta}{b} \ .\
\prod_{k=1}^{n}
E_0\left(e^{-\delta \tau_0^+}, -X_{\tau_0^+ -}\in da_k, X_{\tau_0^+ }\in dr_k,\tau_0^+<\tau_{-y} \right)
,
\end{multline*}
 and finally, letting $y\to\infty$, we get
$$
\PP(N_T=n, A_i'\in da_i,R_i'\in dr_i,i=1,\ldots,n)=\frac{\delta}{b} \ .\
\prod_{k=1}^{n}
E_0\left(e^{-\delta \tau_0^+}, -X_{\tau_0^+ -}\in da_k, X_{\tau_0^+ }\in dr_k \right)
.
$$\end{cor}

\begin{rem}
We observe that conditional on $N_T=n$ and/or conditional on $N_T=n$ and $T<y$, the random pairs $(A_i,R_i)$, $1 \leq i \leq n$, are i.i.d.\ and their common distribution does not depend on $n$.
\end{rem}

\subsection{Completely asymmetric L{\'e}vy processes}

Here, we seek to provide the reader with more explicit formulae regarding the quantities considered in Subsection \ref{subsec:lowerdm}, taking advantage of background knowledge on L{\'e}vy processes.
Except the proposition stated at the end of the present subsection, all results stated here and the references to their original contributors, can be found in \cite{B, B2}.
 
Instead of the jump measure $\pi$ of the L{\'e}vy process $X$ with no negative jumps, it can be convenient to handle its Laplace exponent $\psi$ defined as
\begin{equation}
\label{eqn : psi}
\psi(a):= a -\int_{(0,+\infty]} \pi(dx) (1-e^{-a x}) \qquad a\ge 0.
\end{equation}
Recall that the real number $\pi(\{\infty\})$ can be positive, since particles may have infinite lifetimes. It is also the killing rate of $X$.
The function  $\psi$ is differentiable and convex and we denote by $\eta$ its largest root. Then $\psi$  is increasing on $[\eta,+\infty)$ and we denote by $\phi$ its inverse mapping on this set. 
Furthermore, the so-called two-sided exit problem (exit of an interval from the bottom or from the top by $X$) has a simple solution, in the form
\begin{equation}
\label{eqn : two-sided}
P_s(\tau_0< \tau^+_t) = \frac{W(t-s)}{W(t)},
\end{equation}
where the so-called scale function $W$ is the non-negative, nondecreasing, differentiable function such that $W(0)=1$, characterised by its Laplace transform 
\begin{equation}
\label{eqn : LT scale}
\intgen dx\, e^{-a x} \, W(x) = \frac{1}{\psi(a)} \qquad a>\eta.
\end{equation}
Equation \eqref{eqn : two-sided}  gives the probability that $X$ exits the interval $(0, t]$ from the bottom. The following formula gives the Laplace transform of the first exit time $\rho_t:=\tau_0\wedge \tau_t^+$ on this event. For any $q > 0$,
\begin{equation}
\label{eqn : two-sided-LT}
E_s\left(e^{-q \rho_t},\tau_0< \tau^+_t\right) = \frac{W^{(q)}(t-s)}{W^{(q)}(t)},
\end{equation}
where the so-called $q$-scale function $W^{(q)}$ is the non-negative, nondecreasing, differentiable function such that $W^{(q)}(0) = 1$, characterised by its Laplace transform
\begin{equation}
\label{eqn : LT q-scale}
\intgen dx\, e^{-a x} \, W^{(q)}(x) = \frac{1}{\psi(a)-q} \qquad a>\phi(q).
\end{equation}
The $q$-resolvent of the process killed upon exiting $(0, t]$ is given by the following formula ($s,y\in (0,t]$)
\begin{equation}
\label{eqn : resolvent}
u_t^q(s,y):=E_s\left( \int_0^{\rho_t} e^{-qv} \indic{X_v \in dy}dv\right) /dy
=\frac{W^{(q)}(t-s)W^{(q)}(y)}{W^{(q)}(t)}-\indic{y> s}W^{(q)}(y-s).
\end{equation}
We also need the $q$-resolvent of the process killed upon exiting $(-\infty, 0]$ ($s,y\ge 0$)
\begin{equation}
\label{eqn : resolvent2}
u^q(s,y):=E_{-s}\left( \int_0^{\tau_0^+} e^{-qv} \indic{-X_v \in dy}dv\right) /dy
=e^{-\phi(q)y}\,W^{(q)}(s)-\indic{s> y}W^{(q)}(s-y).
\end{equation}
Last, we have the following expression for the bi-variate law of the undershoot and overshoot on the event that the process exits $(0,t]$ from the top
\begin{equation}
\label{eqn : shoots}
E_s\left( e^{-q\rho_t}, \tau_t^+<\tau_0, X_{\rho_t-}\in dy,  X_{\rho_t}-X_{\rho_t-}\in dz\right) 
= u_t^q(s,y)\, dy\, \pi(dz) \qquad z+y>t, y\in (0,t),
\end{equation}
and the analogue for the exit from $(-\infty,0]$.\\
\\
The next statement deals with the following quantities of interest in relation to  Subsection \ref{subsec:lowerdm}. For any $t>0$ and $q\ge 0$, set
$$
G_q(t):= 1- E_0\left(e^{-q \tau_0^+}, \tau_0^+<\tau_{-t}\right).
$$
In particular, as $t\to\infty$, $G_q(t)$ converges to $G_q(\infty):= E_0\left(1-e^{-q \tau_0^+}\right)$.

\begin{prop} 
\label{prop:Gq}
For any $q,r \ge 0$ and $0<a<t$,
$$
E_0\left( e^{-q \tau_0^+}, -X_{\tau_0^+ -}\in da, X_{\tau_0^+ }\in dr, \tau_0^+ <\tau_{-t} \right)=
\frac{W^{(q)}(t-a)}{W^{(q)}(t)}\, da\, d\pi(a+r),
$$
and
$$
E_0\left( e^{-q \tau_0^+}, -X_{\tau_0^+ -}\in da, X_{\tau_0^+ }\in dr\right)=
e^{-\phi(q)a}\, da\, d\pi(a+r).
$$
For any $q,t\ge 0$, 
$$
G_q(t) = \frac{1+q\int_0^t W^{(q)}(s)\, ds}{W^{(q)}(t) } 
$$
and
$$
\phi(q) = q + b E \left(1-e^{-q \tau_0}\right) = \frac{q}{G_q(\infty)}.
$$
\end{prop}

\begin{proof}
The first two displays stem from instantiating \eqref{eqn : resolvent} (resp. \eqref{eqn : resolvent2}) and \eqref{eqn : shoots} at $s=t$, using the spatial homogeneity of L{\'e}vy processes (resp. $s=0$).

Writing $\bar{\pi}(x):= \pi((x,+\infty])$, $x>0$, we get from the first display
$$
E_0\left( e^{-q \tau_0^+}, -X_{\tau_0^+ -}\in da, \tau_0^+ <\tau_{-t} \right)=
\frac{W^{(q)}(t-a)}{W^{(q)}(t)}\, da\, \bar{\pi}(a) ,
$$
so that 
$$
1-G_q(t) = \frac{g_q(t)}{W^{(q)}(t)},
$$
where
$$
g_q(t):=\int_0^t da\,W^{(q)}(t-a)\, \bar{\pi}(a)\qquad t\ge 0.
$$
To prove the third display, it remains to prove that for all $t\ge 0$,
$$
W^{(q)}(t) - g_q (t) = 1+q\int_0^t W^{(q)}(s)\, ds.
$$
Now for any $\lbd >\phi(q)$, the Laplace transform of the non-negative mapping $h_q:t\mapsto g_q(t) +1+q\int_0^t W^{(q)}(s)\, ds$ evaluated at $\lbd$ is
\debeq
\intgen dt\,e^{-\lbd t}h_q(t) &=& \frac{1}{\lbd} + \frac{1}{\psi(\lbd)-q}\intgen dt\,e^{-\lbd t}\int_{(t,+\infty]}\pi(dr) + q\intgen dt\,e^{-\lbd t}\int_0^t ds\,W^{(q)}(s)\\
 &=& \frac{1}{\lbd} + \frac{1}{\psi(\lbd)-q}
\int_{(0,+\infty]}\pi(dr)\frac{ 1-e^{-\lbd r}}{\lbd} + q\intgen ds\,W^{(q)}(s) \frac{ e^{-\lbd s}}{\lbd} \\
&=& \frac{1}{\lbd}\left\{1  + \frac{\lbd -\psi(\lbd)}{\psi(\lbd)-q} + \frac{q}{\psi(\lbd)-q}\right\} = \frac{1}{\psi(\lbd)-q},
\fineq
which is the Laplace transform of $W^{(q)}$ and characterises it.

The last two equalities are classical results in fluctuation theory of L{\'e}vy processes \cite{B}. To be more specific, the first equality is the well-known fact that the inverse mapping of the Laplace exponent of a L{\'e}vy process without negative jumps is the Laplace exponent of its dual ladder time process. Since $q\mapsto G_q(\infty)= E_0\left(1-e^{-q \tau_0^+}\right)$ is the ladder time process of $X$, the Wiener--Hopf factorisation yields the second equality (which could also be proved in the same fashion as the third display, using the second display).
\end{proof}

\subsection{Summary statement with explicit formulae}

The analytical results of Proposition \ref{prop:Gq} can be applied straightforwardly to rephrase the conceptual results of Subsection \ref{subsec:lowerdm}, at the preference of the reader. The next statement is one of the practical ways of doing this. It provides explicit formulae, up to the knowledge (or numerical computation) of $\delta$-scale functions (occasionally via $G_\delta$, but then use Proposition \ref{prop:Gq}) and $\phi$ (which is fast to compute as the inverse mapping of $\psi$), for various marginals of interest of the splitting tree stopped when the first clock rings.

\begin{prop}
\label{last}
Let $n\ge 1$ and $y,t>0$.
The joint law of $T$ and $N_T$ is given by
$$
\PP(N_T=n, T\in dt)=
-\frac{n\delta}{b} \ .\
G_\delta'(t)\,
\left(1-G_\delta(t)\right)^{n-1}\,dt.
$$
As a consequence,
$$
\PP(N_T=n, T<y)=\frac{\delta}{b} \ 
\left(1-G_\delta(y)\right)^n,
$$
 with respective one-dimensional marginals
 $$
\PP(N_T=n)=\frac{\delta}{b} \ 
\left(1-G_\delta(\infty)\right)^n = 
\frac{\delta}{b} \ 
\left(1-\frac{\delta}{\phi(\delta)}\right)^n
\quad \mbox{ and }\quad\PP(T<y)=\frac{\delta}{b} \ .\
\frac{1-G_\delta(y)}{G_\delta(y)} .
$$
In particular,
$$
\PP(N_T=n\mid T=t)=
nG_\delta^2(t)\,
\left(1-G_\delta(t)\right)^{n-1}\quad \mbox{ and }\quad\PP(N_T=n\mid T<y)=
G_\delta(y)(1-G_\delta(y))^{n-1} .
$$
Also, the probability $p$ that $T<\infty$ equals
$$
p = 
\frac{\delta}{b} \ .\
\frac{1-G_\delta(\infty)}{G_\delta(\infty)}
= \frac{\phi(\delta) -\delta}{b} . 
$$
Conditional on $\{N_T=n, T<y\}$, ($y\le\infty$) the ages and residual lifetimes of the $n$ alive individuals at time $T$ are i.i.d., distributed as the r.v.\ $(A(y),R(y))$ (independent of $n$). If $y<\infty$,
\begin{multline*}
\PP(A(y)\in da, R(y)\in dr) = \frac{1}{1-G_\delta(y)} \ \frac{W^{(\delta)}(y-a)}{W^{(\delta)}(y)}\, da\, d\pi(a+r)\\
=\frac{W^{(\delta)}(y-a)}{W^{(\delta)}(y)-1-\delta\int_0^y W^{(\delta)}(s)\, ds}\, da\, d\pi(a+r).
\end{multline*}
If $y=\infty$,

$$
\PP(A(\infty)\in da, R(\infty)\in dr) = \frac{1}{1-G_\delta(\infty)} \ e^{-\phi(\delta)a}\, da\, d\pi(a+r)
=\frac{\phi(\delta)}{\phi(\delta)-\delta} \ e^{-\phi(\delta)a}\, da\, d\pi(a+r).
$$
\end{prop}

\section{A (more general) model of epidemics}

As in \cite{T,BWTB}, we aim to model the spread of some antibiotic resistant bacteria like MRSA (Methicillin--resistant \textit{Staphylococcus aureus}) in a hospital. Once in a while, a patient is colonised by MRSA (presumably by introduction from outside) and this may cause an outbreak in the hospital. This outbreak is only detected at the first time $T$ when one of the carriers declares herself, i.e., when the first symptoms appear in a carrier, or at the first positive medical exam of a carrier.

We assume that
\begin{itemize}
\item 
patients have i.i.d\  lengths of stay in the hospital, all distributed as some positive random variable $K$ with finite expectation;
\item 
the outbreak starts with the infection of a randomly chosen patient;
\item
the length of stay is not influenced by whether or not an individual carries MRSA (neutrality, or exchangeability assumption);
\item
during an outbreak no further introductions from outside occur (no immigration);
\item
carriers are infective from the first time they were infected till their departure from the hospital;
\item
while infective, patients independently transmit MRSA to other individuals at times of a Poisson process with parameter $b$ (susceptible individuals are always assumed to be in excess, so that effects of the finite size of the hospital are ignored);
\item
as a consequence of the renewal theorem (assuming stationarity of the regenerative set of arrivals at the hospital), the length of stay of a patient conditional on infection is a size-biased version of $K$, and the time at which she is infected is independent, uniformly distributed during her stay;
\item
each patient can be detected to be a carrier only after an independent exponential time with parameter $\delta$ running from the beginning of her infection (time of screening or of developing symptoms in this patient). The first time $T$ when a carrier is detected is called \emph{detection time};
\item
At detection time, all patients in the hospital are screened with a perfect test, so all carriers at $T$ are immediately identified. 
\end{itemize}
\begin{rem}
\label{rem}
The second assumption can be disputable, since MRSA is often introduced by a patient who already carries MRSA before entering the hospital (personal communication with Martin Bootsma). Changing this assumption on introduction of MRSA for a more realistic one would make the analysis harder, although possible, and obscure the illustrative character of the example provided in this section.
\end{rem}

It is not possible to obtain useful data from patients who already left the hospital at the moment of detection. Indeed, most carriers leaving the hospital will soon lose MRSA because - in the absence of antibiotic pressure - the antibiotic resistant strains will soon be outcompeted by antibiotic susceptible strains.
Thus, our goal is to infer the parameters of the epidemics by using available medical data belonging to the detected carriers. 


Thus the model is a Crump--Mode--Jagers branching process where every birth event is interpreted as an infection, and individuals are endowed with i.i.d.\ bivariate r.v.\ distributed as the pair $(U,V)$, with $V$ the lifetime (as an infective), i.e., the time between infection and departure from the hospital, and $U$ the time already spent in the hospital before infection. Individuals ``give birth'' at constant rate $b$ during their (infective) lifetime (length $V$) to copies of themselves. Finally, the joint law of $(U,V)$ is given by
\begin{equation}\label{sizebias}
\EE(f(U,V)) = m^{-1} \int_{(0,\infty)} \PP(K\in dz) \int_{(0,z)} dx\,f(x, z-x),
\end{equation}
where $m:=\EE(K)$ and $f$ is any non-negative Borel function.



At detection time $T$, all carriers $i=1,\ldots,N_T$ are identified and we focus on the following medical data belonging to them 
\begin{itemize}
\item $U_i$ is the time already spent in the hospital by carrier $i$ upon her infection;
\item $A_i$ is the time elapsed between infection of carrier $i$ and $T$ (`age' of infection) ;
\item $R_i$ is the remaining length of stay of carrier $i$ in the hospital after $T$ (`residual lifetime' of infection); 
\item $V_i:=A_i+R_i$ is the total infective lifetime of carrier $i$;
\item $H_i:=U_i+A_i$ is the time elapsed between entrance in the hospital of carrier $i$ and time $T$.
\end{itemize}
Note that $(U_i,V_i)$ is merely the typical pair $(U,V)$ attached to carrier $i$, and that $A_i$ and $R_i$ have the interpretations given in the previous section. The quantities of empirical interest are the random variables $H_i$, which should be easy to obtain from the hospital administrations. Also, the distribution of $K$ should be easy to estimate from hospital data.


 It is not difficult to see that with this extra information, Proposition  \ref{last} still holds with $\mu$ (and hence $\psi$, $\phi$, $W^{(\delta)}$,...) defined thanks to \eqref{sizebias} as 
 \begin{equation}
 \label{V}
\mu(dx) :=\PP(V \in dx) = m^{-1}\,\PP(K>x)\, dx.
 \end{equation}
Actually, we also have the following straightforward extension of Proposition  \ref{last}.
  \begin{cor}
  Conditional on $\{N_T=n\}$, the triples $(U_i, A_i, R_i)$ of the $n$ (randomly labelled)  carriers at time $T$ are i.i.d., distributed as the r.v.\ $(U,A,R)$ (independent of $n$), where
$$
\EE(f(U,A,R)) =\frac{b}{m}\,\frac{\phi(\delta)}{\phi(\delta)-\delta} \ \int_{u=0}^\infty du \int_{a=0}^\infty da \int_{z=u+a}^\infty \PP(K\in dz) \,e^{-\phi(\delta)a}\, f(u,a,z-u-a).
$$
In particular, the times $H_i=U_i+A_i$ spent in the hospital up to time $T$ are i.i.d., distributed as the r.v.\ $H$
\begin{equation}\label{hdist}
\PP(H\in dy) = \frac{b/m}{\phi(\delta)-\delta}\  \PP(K>y) \,\big(1-e^{-\phi(\delta)y}\big)\, dy.
\end{equation}
  \end{cor}
\begin{rem}\label{remH}
From the definition of $\phi(a)$ we deduce that $$\delta = \phi(\delta) -b\int_0^{\infty} \mu(dx) (1-e^{-\phi(\delta)x}) \Leftrightarrow \frac{b}{\phi(\delta)-\delta} = \frac{1}{\int_0^{\infty} \mu(dx) (1-e^{-\phi(\delta)x})}$$
and (\ref{hdist}) might be rewritten as
$$
\PP(H\in dy) = \frac{1/m}{\int_0^{\infty} \mu(dx) (1-e^{-\phi(\delta)x})}\  \PP(K>y) \,\big(1-e^{-\phi(\delta)y}\big)\, dy.
$$
Finally filling in (\ref{V}) gives
\begin{equation}\label{hdistdif}
\PP(H\in dy) = \frac{\PP(K>y) \,\big(1-e^{-\phi(\delta)y}\big)\, dy}{\int_0^{\infty} \PP(K>x) (1-e^{-\phi(\delta)x})dx}.
\end{equation}
The rhs depends on $K$ (which might be estimated from independent hospital data) and $\phi(\delta)$ only.
\end{rem}

Now assume that various outbreaks in various hospitals are observed at their detection times. If the sizes of outbreaks (all distributed as $N_T$) are the only observable statistics, then, as stressed in \cite{BWTB,T}, the fact that $N_T$ is geometrically distributed only allows for the estimation of a single epidemiological parameter.
Enlarging this information to, e.g., the times $H_i$ spent in the hospital before $T$, we can hope to make finer inferences on the dynamical characteristics of those epidemics.

Assume that $n$ outbreaks are observed of sizes $x_1, x_2, \cdots, x_n \in \mathbb{N}_{>0}$ and $s(n) := \sum_{i=1}^n x_i$ carriers are detected, which at the time of detection have been in the hospital for $y_1,y_2,\cdots, y_{s(n)} \in \mathbb{R}_+$ time units. We also assume that, since the distribution of $K$ may be estimated from independent hospital data, its distribution is known exactly.

Using Remark \ref{remH} the likelihood of the observations $L(b,\delta;x_1,\cdots,x_n, y_1, \cdots y_{s(n)})$ is given by 
\begin{multline}\label{likeli}
L(b,\delta) = \left(\prod_{i=1}^n \frac{\delta}{\phi(\delta)} \left(1-\frac{\delta}{\phi(\delta)}\right)^{x_i-1}  \right) \
\prod_{j=1}^{s(n)} \PP(H\in dy_j) \\
= \left(\frac{\delta}{\phi(\delta)}\right)^n \left(1-\frac{\delta}{\phi(\delta)}\right)^{{s(n)}-n} \
\prod_{j=1}^{s(n)} \frac{\PP(K>y_j) \,\big(1-e^{-\phi(\delta)y_j}\big)\, dy_j}{\int_0^{\infty} \PP(K>x) (1-e^{-\phi(\delta)x})dx} .
\end{multline}
We write  $L= L_1L_2$, where
$$L_1(b,\delta) = \left(\frac{\delta}{\phi(\delta)}\right)^n \left(1-\frac{\delta}{\phi(\delta)}\right)^{{s(n)}-n}$$
and
$$L_2(b,\delta) = \prod_{j=1}^{s(n)} \frac{\PP(K>y_j) \,\big(1-e^{-\phi(\delta)y_j}\big)\, dy_j}{\int_0^{\infty} \PP(K>x) (1-e^{-\phi(\delta)x})dx} .$$
We observe that $L_1$ is the likelihood function for $n$ realisations $x_1,x_2,\cdots, x_n$ of i.i.d\ geometric random variables with parameter $g_1:=g_1(b,\delta):= \delta/\phi(\delta)$, while $L_2$ (assuming that the distribution of $K$ is exactly known) only depends on $g_2:=g_2(b,\delta):=\phi(\delta)$.
Observe that $\delta = g_1g_2$ and (recalling $m:=\EE(K)$), 
$$b =  \frac{\phi(\delta)-\delta}{\int_0^{\infty} \mu(dx) (1-e^{-\phi(\delta)x})} = \frac{m g_2(1-g_1)}{\int_0^{\infty} \PP(K>x)\, (1-e^{-g_2x})dx }.$$ 
Furthermore, reparametrization of $L_1(b,\delta)$ in a function of $g_1$ and $g_2$  results in a function which is independent of $g_2$, while reparametrization of  $L_2(b,\delta)$ 
in a function of $g_1$ and $g_2$ results in a function which is independent of $g_1$. 
It is straightforward to deduce that the maximum likelihood estimator (MLE) of $g_1$, say $\hat{g}_1$, is given by  $$\hat{g}_1=\frac{n}{s(n)},$$ while, since $L_1$ does not depend on $g_2$, the MLE of $g_2$, say $\hat{g}_2$, is given by 
$$
\hat{g}_2 = \mbox{arg}\max_{g_2} \prod_{j=1}^{s(n)} \frac{\PP(K>y_j) \,\big(1-e^{-g_2y_j}\big)\, dy_j}{\int_0^{\infty} \PP(K>x) (1-e^{-g_2x})dx} .
$$
Standard theory on maximum likelihood estimation gives that the MLEs of $b$, say $\hat{b}$,  and $\delta$, say $\hat{\delta}$, are given by 
$$
\hat{b}=\frac{m \hat{g}_2(1-\hat{g}_1)}{\int_0^{\infty} \PP(K>x)\, (1-e^{-\hat{g}_2x})dx}\quad\mbox{ and }\quad\hat{\delta}=\hat{g}_1 \hat{g}_2.$$

If $H$ is exponentially distributed with parameter $\nu$ \cite{BWTB,T}, then $\hat{g}_1 = n/s(n)$, while 
\begin{multline*}
\hat{g}_2 = \mbox{arg}\max_{g_2} \prod_{j=1}^{s(n)} \left(\frac{(\nu + g_2) \nu}{g_2} (1- e^{-g_2y_j})\, e^{- \nu y_j} \right) = \\
\mbox{arg}\max_{g_2}\left( s(n) \log\left(1+\frac{\nu  }{g_2}\right) + \sum_{j=1}^{s(n)} \log\big(1- e^{-g_2 y_j}\big) \right)
\end{multline*}
and the MLE of $b$ and $\delta$ are given by $\hat{b}=(1-\hat{g}_1)(\nu + \hat{g}_2)$ and $\hat{\delta}=\hat{g}_1 \hat{g}_2$.

Note that it is possible to allow for differences in the distributions of lengths of stay (the random variable $K$) and infection rates (the parameter $b$) for different hospitals, while keeping the biologically governed rate of onset of symptoms ($\delta$) the same for all hospitals. In that case we use the likelihood (\ref{likeli}) with hospital specific parameters and observations, for estimation.

Derivation of similar formulae for models relaxing too simplistic assumptions (see Remark \ref{rem}), and applications  to real hospital data,
will be addressed in a future work.


\paragraph{Acknowledgements.} A.L. is funded by project MANEGE `Mod\`eles
Al\'eatoires en \'Ecologie, G\'en\'etique et \'Evolution'
09-BLAN-0215 of ANR (French national research agency).
P.T.\ is funded by Vetenskapsr{\aa}det (Swedish research counsel) projectnr.\ 2010--5873.

\end{document}